\documentclass[11pt]{article}

\usepackage{mathrsfs}
\usepackage{amsfonts}
\usepackage{titlesec}
\usepackage{enumerate}
\usepackage[toc,page,title,titletoc,header]{appendix}
\usepackage{amsmath,amsthm,amssymb,anysize}
\usepackage[numbers,sort&compress]{natbib}
\usepackage{color}
\usepackage{graphicx}

\theoremstyle{definition}
\newtheorem{thm}{Theorem}
\newtheorem{lem}{Lemma}[section]
\newtheorem{rem}[lem]{Remark}
\newtheorem{prop}[lem]{Proposition}

\newtheorem{clm}{Claim}
\newtheorem{exa}{Example}
\newtheorem{case}{Case}
\newcommand{\periodafter}[1]{#1.}
\titleformat{\subsection}[runin]
{\normalfont\bfseries}{\thesubsection}{0.5em}{\periodafter}
\marginsize{3cm}{3cm}{2.5cm}{2.2cm}
\numberwithin{equation}{section}

\def\W{\mathcal{W}}

\newcommand{\C}{\mathbb{C}}

\newcommand{\N}{\mathbb{N}}
\newcommand{\Z}{\mathbb{Z}}

\def\a{\alpha}
\def\b{\beta}
\def\l{\lambda}

\def\HVir{\mathcal{L}}
\def\Vir{\mathrm{Vir}}

\title{Modules over the Heisenberg-Virasoro and $W(2,2)$ algebras}
\author{Hongjia Chen and Xiangqian Guo}
\date{ }

\begin{document}

\maketitle

\begin{abstract} In this paper, we consider the modules for the
Heisenberg-Virasoro algebra and the W algebra $W(2,2)$. We determine
the modules whose restriction to the Cartan subalgebra (modulo
center) are free of rank $1$ for these two algebras. We also
determine the simplicity of these modules. These modules provide new
simple modules for the W algebra $W(2,2)$.
\end{abstract}


\vskip 5pt \noindent {\em 2000  Math. Subj. Class.:} 17B10, 17B20, 17B65, 17B66, 17B68

\vskip 10pt
\section{Introduction}

Let $\C, \Z, \Z_+, \N$ be the sets of all complexes, all integers,
all non-negative integers and all positive integers, respectively.
The \textbf{Virasoro algebra} $\Vir$ is an infinite dimensional Lie
algebra over the complex numbers $\C$, with basis $\{L_n,C\,\, |\,\,
n \in \Z\}$ and defining relations
$$
[L_{m},L_{n}]=(n-m)L_{n+m}+\delta_{n,-m}\frac{m^{3}-m}{12}C, \quad
m,n \in \Z,
$$
$$
[C, L_{m}]=0, \quad m \in \Z,
$$ which is the universal central extension
of the so-called infinite dimensional \textbf{Witt algebra} of rank
$1$.

The theory of weight Virasoro modules with finite-dimensional weight
spaces is fairly well developed (see \cite{KR} and references
therein). In 1992, O. Mathieu \cite{M} classified all simple
Harish-Chandra modules, that is, simple modules with
finite-dimensional weight spaces, over the Virasoro algebra, which
is conjectured by Kac \cite{K}. Recently, many authors constructed
many classes of simple non-Harish-Chandra modules, including simple
weight modules with infinite-dimensional weight spaces ( see
\cite{CGZ, CM, LLZ, LZ2, Z}) and simple non-weight modules (see
\cite{BM, LGZ, LLZ, LZ1, MW, MZ1, OW, TZ1, TZ2, TZ3, TZ4}).

For $\lambda \in \C^*=\C\setminus\{0\}, \a\in \C,$ denote by
$\Omega(\lambda,\a)=\C[t]$ the polynomial algebra over $\C$. In
\cite{LZ1} the $\Vir$-module structure on $\Omega(\lambda,\a)$ is
given by
$$
C t^i=0, L_m t^i=\lambda^m (t-ma)(t-m)^i,\quad m\in \Z, i\in \Z_+.
$$
From \cite{LZ1} we know that $\Omega(\lambda,\a)$ is simple if and
only if $\lambda,\a\in \C^*.$ If $\a=0,$ then $\Omega(\lambda,0)$
has an simple submodule $t \Omega(\lambda,0)$ with codimension $1.$

Recently Tan and Zhao \cite{TZ4} showed that the above defined
modules $\Omega(\l,\a)$ are just those $\Vir$-modules that are free
of rank $1$ when restricted to the subalgebra $\C L_0$. In fact they
did similar work for the Witt algebras of all ranks: classifying all
Witt algebra modules that are free of rank $1$ when restricted to
the Cartan subalgebra of the Witt algebra. This kind of problems
originated from an earlier work \cite{N} of J. Nilson who determined
the $\mathfrak{sl}_n$-modules that are free of rank $1$ when
restricted to the Cartan subalgebra of $\mathfrak{sl}_{n}$, where
$n\geq 2$ is a positive integer.

In the present paper, we will follow their ideas to consider similar
questions for the Heisenberg-Virasoro algebra and the W algebra
$W(2,2)$. The Heisenberg-Virasoro algebras was introduced and
studied in \cite{ACKP} and for more results for this algebra please
refer to \cite{B, LZ3, CG} and references therein. The W algebra
$W(2,2)$ was first introduced and studied in \cite{ZD} and for more
information about this algebras please refer to \cite{GLZ, LiZ}.

\section{Representations over the Heisenberg-Virasoro algebra}

The \textbf{Heisenberg-Virasoro algebra $\HVir$} is the universal
central extension of the Lie algebra $\{f(t)\frac{d}{dt} + g(t)\ |\
f, g\in \C[t^{\pm 1}]\}$ of differential operators of order at most
one. More precisely, it is the complex Lie algebra that has a basis
$\{\,L_n,\,I_n,\,C_1,C_2, C_3 \,|\,n\in \Z\,\}$ subject to the
following Lie brackets:
\begin{equation*}
\begin{split}
[L_m,L_n] & =(n-m)L_{m+n}+\frac{m^3-m}{12}\delta_{m+n,0}C_1\\
[L_m,I_n] &=nI_{m+n}+\delta_{m+n,0}(m^2+m)C_2  \\
[I_m,I_n] &=m\delta_{m+n,0}C_3
\end{split}\quad\quad \forall\ m,n\in\Z
\end{equation*}
where $C_1, C_2, C_3$ and $I_0$ span the center of $\HVir$. The Lie
subalgebra spanned by $\{L_n, C_1 \,|\, n \in \Z\}$ is just the
Virasoro algebra.

In this section we will determine the Heisenberg-Virasoro algebra
modules which are free of rank $1$ when regarded as a $\C
L_0$-module. For convenience, we cite the following result in
\cite{TZ4}.

\begin{thm}  \label{Vir}
Let $M$ be a $U(\Vir)$-module such that the restriction of $U(\Vir)$
to $U(\C L_0)$ is free of rank $1$, that is, $M=U(\C L_0)v$ for some
torsion-free $v\in M$. Then $M\cong \Omega(\lambda,\a)$ for some
$\a\in \C, \lambda\in \C^*$. Moreover, $M$ is simple if and only if
$M\cong \Omega(\lambda,\a), \lambda,\a\in \C^*$.
\end{thm}

The $\Vir$-module $\Omega(\l,\a)$ can be naturally made into an
$\HVir$-module.

\begin{exa}\label{ex-HVir}
For $\lambda \in \C^*,\a,\b \in \C,$ denote by
$\Omega(\lambda,\a,\b)=\C[t]$ the polynomial algebra over $\C$. We
can define the $\HVir$-module structure on $\Omega(\lambda,\a,\b)$
as follows
$$
C_j f(t)=0, \; L_m f(t)=\lambda^m (t-m\a)f(t-m), \; I_m f(t)=\b
\lambda^m  f(t-m)
$$
where $f(t) \in \C[t]$, $m\in \Z$ and $j=1,2,3$. From \cite{LZ1} it
is easy to know that $\Omega(\lambda,\a,\b)$ is simple if and only
if $\a\neq0$ or $\b\neq0$. If $\a=\b=0$, then $\Omega(\lambda,0,0)$
has an simple submodule $t \Omega(\lambda,0,0)$ with codimension
$1.$
\end{exa}

Now we have the main result of this section:

\begin{thm}Let $M$ be a $U(\HVir)$-module such that the restriction of $U(\HVir)$ to $U(\C L_0)$ is free of rank $1$.
Then $M \cong \Omega(\lambda,\a,\b)$ for some $\a, \b\in \C,
\lambda\in \C^*$. Moreover $M$ is simple if and only if $M\cong
\Omega(\lambda,\a,\b)$ for some $\lambda\in\C^*, \a, \b\in \C$ with
$\a\neq 0$ or $\b\neq 0$.
\end{thm}

\begin{proof} The second result follows from the first result and
Example \ref{ex-HVir}. We now prove the first result.

Denote the action of $I_0$, $C_1$, $C_2$, $C_3$ by $c_0$, $c_1$,
$c_2$, $c_3$, respectively. By Theorem \ref{Vir}, we have $M\cong
\Omega(\lambda,\a)=\C[t]$ as $\Vir$-modules, where $\l, \a\in\C,
\l\neq 0$ and $t$ is an indeterminant. Hence we can assume that
$c_1=0$ and
$$
L_m f(t)=\lambda^m (t-m\a)f(t-m),\quad m\in \Z, f(t) \in \C[t].
$$

Now we consider the action of $I_m$'s. First we can easily get that
\begin{equation}  \label{claim1}
I_m(f(t))=I_m(f(L_0)(1))=f(t-m)I_m(1),\quad\forall\  m\in \Z.
\end{equation}

\begin{case}
$I_m(1) = 0$ for some $m \in \Z^*$.
\end{case}

By \eqref{claim1} we know that $I_m(M)=0$. Now following from the
defining relations of $\HVir$, we have $I_n(M) = 0$ for all $n \in
\Z^*$ and $c_0=c_2=c_3=0$. In this case $M\cong \Omega(\l, \a, 0)$.

\begin{case}
$I_m(1) \neq 0$ for all $m \in \Z^*$.
\end{case}

Assume that
\begin{equation*}
I_m(1) = f_m(t) = \sum_{i=0}^{k_m}b_{m,i} t^i,\quad\forall\ m\in\Z,
\end{equation*}
where $b_{m,i} \in \C$ and $b_{m,k_m} \neq 0$. Now we calculate
$[L_{-m},I_m](1)$ as follows
\begin{equation*}
\begin{split}
[L_{-m},I_m](1)
= & L_{-m}I_m(1)-I_mL_{-m}(1)  \\
= & L_{-m}f_m(t)-I_m\lambda^{-m} (t+m\a) \\
= & \lambda^{-m} (t+m\a)f_m(t+m)-\lambda^{-m} (t-m+m\a)f_m(t) \\
= & \lambda^{-m}m(k_m+1)b_{m,k_m}t^{k_m} + \text{ lower-degree terms
w.r.t. } t,
\end{split}
\end{equation*}
which implies that $k_m=0$ for all $m \in \Z$. Hence we have
$I_m(1)=a_m \in \C$ for all $m \in \Z$ with $a_0=c_0$. Now
$[I_m,I_n](1)=m\delta_{m+n,0}c_3$ implies that
\begin{equation}
c_3=[I_1,I_{-1}](1) = a_1a_{-1}- a_{-1} a_1=0
\end{equation}
and
$[L_m,I_n](1) =nI_{m+n}(1)+\delta_{m+n,0}(m^2+m)c_2 $ implies
\begin{equation}  \label{am-c2}
na_{m+n}+\delta_{m+n,0}(m^2+m)c_2=a_n\l^m(t-m\a)-a_n\l^m(t-n-m\a)=na_n\l^m.
\end{equation}
Taking $m=-1$ in \eqref{am-c2}, we have $a_n=\l a_{n-1}$ for all $n \neq 0$.
Hence $c_2=0$.
Then checking the recurrence relations $na_{m+n}=na_n\l^m$,
we have that $a_m=c_0 \l^m$.

The above discussion allows us to establish an $\HVir$-module homomorphism
\begin{equation*}
\begin{split}
M \longrightarrow \Omega(\lambda,\a,c_0),\quad\quad t^i \mapsto t^i.
\end{split}
\end{equation*}
Clearly, this is an $\HVir$-module isomorphism. 
This completes the proof.
\end{proof}

\section{Representations over the $W$-algebra $W(2,2)$}

Let $\W$ denote the complex \textbf{Lie algebra $W(2,2)$}, that has
a basis $\{\,L_n,\,W_n,\,C_1,C_2\,|\,n\in \Z\,\}$ and the Lie
brackets defined by:
\begin{equation*}
\begin{split}
[L_m,L_n] & =(n-m)L_{m+n}+\frac{m^3-m}{12}\delta_{m+n,0}C_1\\
[L_m,W_n] &=(n-m)W_{m+n}+\frac{m^3-m}{12}\delta_{m+n,0}C_2  \\
[W_m,W_n] &=0
\end{split}
\end{equation*}
where $C_1$ and $C_2$ are central in $\W$. The Lie subalgebra
spanned by $\{L_n, C_1 \,|\, n \in \Z\}$ is the Virasoro algebra.

\subsection{Rank $1$ free $U(\C L_0)$-modules}

The $\Vir$-module $\Omega(\l,\a)$ can be made into a $\W$-module:

\begin{exa}
For $\lambda \in \C^*, \a \in \C,$ denote by
$\Omega_{\W}(\lambda,\a)=\C[t]$ the polynomial algebra over $\C$. We
can define the $\W$-module structure on $\Omega_{\W}(\lambda,\a)$ as
follows
$$
C_j f(t)=W_m f(t)=0, \; L_m f(t)=\lambda^m (t-m\a)f(t-m)
$$
where $f(t) \in \C[t]$, $m\in \Z$ and $j=1,2$. From \cite{LZ1} it is
easy to know that $\Omega_{\W}(\lambda,\a)$ is simple if and only if
$\a\neq0$. If $\a=0$, then $\Omega_{\W}(\lambda,0)$ has an simple
submodule $t \Omega_{\W}(\lambda,0)$ with codimension $1.$

\end{exa}

In this subsection, we will classify the modules over the Lie algebra $\W$
such that the modules are free $U(\C L_0)$-modules of rank $1$. We have

\begin{thm} Let $M$ be a $U(\W)$-module such that the restriction of $U(\W)$ to $U(\C L_0)$ is free of rank $1$.
Then $M \cong \Omega_{\W}(\lambda,\a)$ for some $\a\in \C,
\lambda\in \C^*$. Moreover, $M$ is simple if and only if $M\cong
\Omega_{\W}(\lambda,\a), \lambda, \a\in \C^*$.
\end{thm}

\begin{proof}
Denote the action of $C_1$, $C_2$ by $c_1$, $c_2$, respectively. By
Theorem \ref{Vir}, $M\cong \Omega(\lambda,\a)=\C[t]$ as
$\Vir$-modules for some $\l\in\C^*$ and $\a\in\C$. Hence we can
assume that $c_1=0$ and
$$
L_m f(t)=\lambda^m (t-m\a)f(t-m), m\in \Z, f(t) \in \C[t].
$$

Now we consider the action of $W_m$'s. First we can easily get that
\begin{equation}  \label{claim2}
W_m(f(t))=W_m(f(L_0)(1))=f(t-m)W_m(1),\quad\forall\ m\in \Z.
\end{equation}

\setcounter{case}{0}

\begin{case}
$W_m(1) = 0$ for some $m \in \Z$.
\end{case}

By \eqref{claim2} we know that $W_m(M)=0$. Now following from the
definition relations of $\W$, we have $W_n(M) = 0$ for all $n \in
\Z$ and $c_2=0$. We see that $M$ is actually a $\Vir$-module and
hence $M\cong\Omega_{\W}(\l,\a)$ in this case.

\begin{case}
$W_m(1) \neq 0$ for all $m \in \Z$.
\end{case}

Assume that
\begin{equation*}
W_m(1) = f_m(t) = \sum_{i=0}^{k_m}b_{m,i} t^i
\end{equation*}
where $b_{m,i} \in \C$ and $b_{m,k_m} \neq 0$. For all $mn \leq 0$, we have
\begin{equation*}
\begin{split}
0= & [W_m,W_n](1) \\
= & W_mW_n(1)-W_nW_m(1)   \\
= & W_mf_n(t)-W_nf_m(t) \\
= & f_n(t-m)f_m(t)-f_n(t)f_m(t-n)  \\
= & b_{m,k_m}b_{n,k_n}(nk_m-mk_n) t^{k_m+k_n-1} + \text{
lower-degree terms w.r.t. } t.
\end{split}
\end{equation*}
Hence $k_m=0$ for all $m \in \Z$, i.e., $W_m(1)=a_m$ for some
$a_m\in\C^*$. Now
$$[L_m,W_n](1) =(n-m)W_{m+n}(1)+\delta_{m+n,0}\frac{m^3-m}{12}c_2 $$
implies
\begin{equation}  \label{2am-c2}
(n-m)a_{m+n}+\delta_{m+n,0}\frac{m^3-m}{12}c_2=a_n\l^m(t-m\a)-a_n\l^m(t-n-m\a)=na_n\l^m.
\end{equation}
Taking $m=-1$ and $m=1$ in \eqref{am-c2} respectively, we have
\begin{equation*}
(n+1)a_{n-1}=na_n\l^{-1}   \text{ and } (n-1)a_{n+1}=na_n\l.
\end{equation*}
Hence $a_n=0$, which is a contradiction. This case does not occur.
%
%
\end{proof}

\subsection{Rank $1$ free $U(\C L_0\oplus \C W_0)$-modules}


In this subsection, we will classify the modules over the $W(2,2)$
algebra $\W$ such that the modules are free $U(\C L_0 \oplus \C
W_0)$-modules of rank $1$. Before present the main result, we first
construct some modules with this property.

Fix any $\lambda \in \C^*$ and $\a \in \C$. For any $k \in \N$ and
$n \in \Z$, define the following polynomials
\begin{equation}
h_{n,k;\a}(t) = nt^k -n(n-1)\a \frac{t^k-\a^k}{t-\a}  \in \C[t]
\end{equation}
and let $\mathcal{H}_{\a}$ be the set of families of polynomials
given by
\begin{equation}\label{calH_a}
\mathcal{H}_{\a} = \Big\{\big\{h_n(t)\big\}_{n \in \Z}  \,|\,
h_n(t)=\sum_{i=0}^{+\infty} \xi_ih_{n,i;\a}(t) \in \C[t], \xi_i \in
\C \Big\}.
\end{equation}
Note that $\xi_i$'s are independent of the choice of $n$. In
particular, we have
\begin{equation}
\mathcal{H}_0 = \Big\{\big\{nh(t)\big\}_{n \in \Z}  \,|\, h(t) \in \C[t]\Big\}.
\end{equation}

For any $\mathbf{h}=\big\{h_n(t)\big\} \in \mathcal{H}_{\a}$, denote
by $\Omega(\l,\a,\mathbf{h})=\C[t,s]$ the polynomial algebra over
$\C$. We define the action of $\W$ on $\Omega(\l,\a,\mathbf{h})$ as
follows
\begin{equation} \label{actCW}
C_1\big(f(t,s)\big)=C_2\big(f(t,s)\big)=0, \;\; W_m\big(f(t,s)\big)
= \l^m(t-m\a)f(t,s-m)
\end{equation}
and
\begin{equation} \label{actL}
L_m\big(f(t,s)\big) = \l^m(s+h_m(t))f(t,s-m) -m \l^m
(t-m\a)\frac{\partial}{\partial t}\big(f(t,s-m)\big).
\end{equation}
Then we have

\begin{prop} \label{construction}
$\Omega(\l,\a,\mathbf{h})$ is a $\W$-module for any $\l \in \C^*$,
$\a \in \C$ and $\mathbf{h}=\big\{h_n(t)\big\} \in \mathcal{H}_\a$
under the actions of \eqref{actCW} and \eqref{actL} and
$\Omega(\l,\a,\mathbf{h})$ is free of rank $1$ as a module over
$U(\C L_0 \oplus \C W_0)$. Moreover, we have that $\Omega(\l,
\a,\mathbf{h})$ is simple if and only if $\a \neq 0$.
\end{prop}

\begin{proof}
For any $m,n \in \Z$, we have that
\begin{equation*}
\begin{split}
L_mL_n\big(f(t,s)\big) = & L_m\Big(\l^n(s+h_n(t))f(t,s-n)-n\l^n(t-n\a)\frac{\partial}{\partial t} f(t,s-n)\Big) \\
= & \l^{m+n}(s+h_m(t))(s-m+h_n(t))f(t,s-m-n) \\
  & -n\l^{m+n}(s+h_m(t))(t-n\a)\frac{\partial}{\partial t} f(t,s-m-n)\\
  & - m\l^{m+n}(t-m\a)f(t,s-m-n)\frac{\partial}{\partial t}h_n(t)\\
  & -m\l^{m+n}(t-m\a)(s-m+h_n(t))\frac{\partial}{\partial t}f(t,s-m-n)\\
  & + mn\l^{m+n}(t-m\a)\frac{\partial}{\partial t} f(t,s-m-n)\\
  & + mn\l^{m+n}(t-m\a)(t-n\a)\frac{\partial^2}{\partial^2 t} f(t,s-m-n),\\
\end{split}
\end{equation*}
\begin{equation*}
\begin{split}
L_mW_n\big(f(t,s)\big) = & L_m\big(\l^n(t-n\a)f(t,s-n)\big) \\
= &  \l^{m+n} (t-n\a)(s+h_m(t))f(t,s-m-n) -m  \l^{m+n} (t-m\a)  f(t,s-m-n)  \\
& -m  \l^{m+n} (t-m\a)(t-n\a) \frac{\partial}{\partial
t}\big(f(t,s-m-n)\big)
\end{split}
\end{equation*}
and
\begin{equation*}
\begin{split}
W_nL_m\big(f(t,s)\big) = & W_n\Big(\l^m(s+h_m(t))f(t,s-m) -m \l^m (t-m\a) \frac{\partial}{\partial t}\big(f(t,s-m)\big)\Big)\\
= & \l^{m+n}(t-n\a)(s-n+h_m(t))f(t,s-m-n) \\
& -m \l^{m+n} (t-n\a) (t-m\a) \frac{\partial}{\partial
t}\big(f(t,s-m-n)\big).
\end{split}
\end{equation*}
Hence for any $m,n\in\Z$, we have

\begin{equation*}
\begin{split}
[L_m,L_n]\big(f(t,s)\big) = & \l^{m+n}\Big((s+h_m(t))(s-m+h_n(t))-(s+h_n(t))(s-n+h_m(t))\Big)f(t,s-m-n) \\
   & + \l^{m+n}\Big(n(t-n\a)\frac{\partial}{\partial t}h_m(t)-m(t-m\a)\frac{\partial}{\partial t}h_n(t)\Big)f(t,s-m-n)\\
   & +\l^{m+n}\Big(m(s+h_n(t))(t-m\a)-n(s+h_m(t))(t-n\a)\Big)\frac{\partial}{\partial t} f(t,s-m-n)\\
   & + \l^{m+n}\Big(n(t-n\a)(s-n+h_m(t))-m(t-m\a)(s-m+h_n(t))\Big)\frac{\partial}{\partial t}f(t,s-m-n)\\
   & + mn\l^{m+n}\Big((t-m\a)-(t-n\a)\Big)\frac{\partial}{\partial t} f(t,s-m-n)\\
=  & (n-m)\l^{m+n}\Big(s+(m+n)t^k-(m+n)(m+n-1)\a\frac{t^k-\a^k}{t-\a}\Big)f(t,s-m-n)\\
   & -(n-m)(m+n)\l^{m+n}(t-(m+n)\a)\frac{\partial}{\partial t} f(t,s-m-n)\\
=  & (n-m)L_{m+n}(f(t,s)),
\end{split}
\end{equation*}
where we use the following identity
\begin{equation*}
n(t-n\a)h'_{m}(t)-m(t-m\a)h'_n(t)=-(n-m)mn\a\frac{t^k-\a^k}{t-\a},\quad\
\forall\ m,n\in\Z.
\end{equation*}
Similarly, we can deduce that
\begin{equation*}
\begin{split}
[L_m,W_n]\big(f(t,s)\big)
= & n \l^{m+n}(t-n\a)f(t,s-m-n) -m  \l^{m+n} (t-m\a)  f(t,s-m-n) \\
= & (n-m)\l^{m+n}\big(t-(m+n)\a\big)f(t,s-m-n) \\
= & (n-m)W_{m+n}\big(f(t,s)\big).
\end{split}
\end{equation*}
Finally, for any $m,n \in \Z$, we have that
\begin{equation} \label{checkWW}
W_mW_n\big(f(t,s)\big)= \l^{m+n}(t-n\a)(t-m\a)f(t,s-m-n),
\end{equation}
hence $[W_m,W_n]\big(f(t,s)\big)=0$.

\setcounter{case}{0} For the simplicity, we have

\begin{case}
$\a\neq 0.$
\end{case}
Suppose that $N$ is a nonzero submodule of
$\Omega(\l,\a,\mathbf{h})$. Let $f(t,s)$ be a nonzero polynomial in
$N$ with the smallest $\deg_s\big(f(t,s)\big)$. By \eqref{checkWW},
we have
\begin{equation*}
\begin{split}
&(2W_0W_m-W_{m+1}W_{-1}-W_{m-1}W_1)\big(f(t,s)\big)=2\l^m
\a^2 f(t,s-m), \\
\end{split}
\end{equation*}
forcing $f(t,s-m) \in N$. This implies that
$\deg_s\big(f(t,s)\big)=0$, i.e., $f(t,s)\in \C[t]$. Denote
$f(t)=f(t,s)$. From the actions of $W_0$ and $L_0$, we see that
$f(t) \C[t,s] \subseteq N$. Hence by \eqref{actL} we get
$(t-m\a)f'(t) \in N$ for all $m \in \Z^*$. This immediately gives
that $N=\C[t,s]$, i.e., $\Omega(\l,\a,\mathbf{h})$ is simple.

\begin{case}
$\a=0.$
\end{case}
It is clear that for any $i \in \Z_+$, $t^i\C[t,s]$ is a submodule
of $\Omega(\l,0,\mathbf{h})$ for all $\l\in \C^*$ and
$\mathbf{h}=\big\{nh(t)\big\}_{n \in \Z} \in \mathcal{H}_0$ with
$h(t)\in\C[t]$. 
Moreover, the quotient module $t^{i}\C[t,s]/t^{i+1}\C[t,s]$ is
simple if and only if $h(0) \neq i$. Indeed, for any $0 \neq
t^{i}g(s) \in t^{i}\C[t,s]/t^{i+1}\C[t,s]$, we have
$W_m\big(t^{i}g(s)\big)=0$ and
\begin{equation*}
\begin{split}
L_m\big(t^{i}g(s)\big) \equiv & \l^m (s+mh(0))t^{i}g(s-m) -\l^m mi t^{i} g(s-m)  \\
\equiv & \l^m(s+m(h(0)-i))t^{i}g(s-m),\quad\mod t^{i+1}\C[t,s].
\end{split}
\end{equation*}
Hence $t^{i}\C[t,s]/t^{i+1}\C[t,s] \cong \Omega_{\W}(\l,i-h(0))$ as
$\W$-modules.
\end{proof}

The following is our main result of this subsection:

\begin{thm}Let $M$ be a $U(\W)$-module such that the restriction to $U(\C L_0 \oplus \C W_0)$ is free of rank $1$.
Then $M \cong \Omega(\l,\a,\mathbf{h})$ for some $\a\in \C$,
$\lambda\in \C^*$ and $\mathbf{h}=\big\{h_n(t)\big\} \in
\mathcal{H}_\a$.
\end{thm}

\begin{proof}
Denote the action of $C_1$, $C_2$ by $c_1$, $c_2$, respectively.
Assume that $M=U(\C L_0\oplus \C W_0)v$ for $v \in M$. We divide the
proof into several claims.

\begin{clm}\label{claim-W_m(u)}
For any $u \in M$, $i \in \Z_+$ and $m \in \Z$, we have
\begin{equation} \label{relW}
W_m(W_0^iu) = W_0^iW_mu \text{ and } W_m(L_0^iu) = (L_0-m)^iW_mu.
\end{equation}
\end{clm}

\setcounter{case}{0}

\begin{clm}
$W_m(v) \neq 0$ for all $m \in \Z$.
\end{clm}

Suppose on the contrary that $W_m(v) = 0$ for some $m \in \Z^*$. By
\eqref{relW} we know that $W_m(M)=0$. Now for any $n \neq 0, 2m$, we
have
\begin{equation*}
W_nM=(2m-n)^{-1}[L_{n-m},W_m]M=0.
\end{equation*}
Moreover, we have
\begin{equation*}
W_{2m}M=(4m)^{-1}[L_{-m},W_{3m}]M=0.
\end{equation*}
Hence, for any $n \in \N$, we get
\begin{equation*}
-2nW_0v + \frac{n^3-n}{12}c_2v=[L_n,W_{-n}]v=0.
\end{equation*}
This implies that $W_0v=0$, a contradiction. This claim is true.

\begin{clm}
$W_m(v) =a_m(W_0)$ for some $a_m \in \C[W_0]$.
\end{clm}

Assume that
\begin{equation*}
W_m(v) 
= \sum_{i=0}^{k_m}b_{m,i} L_0^iv
\end{equation*}
where $b_{m,i}=b_{m,i}(W_0) \in \C[W_0]$ and $b_{m,k_m} \neq 0$. For
all $mn\leq 0$, we have
\begin{equation*}
\begin{split}
0= & [W_m,W_n](v)= W_mW_n(v)-W_nW_m(v)   \\
= & \left(\sum_{i=0}^{k_n}b_{n,i}(L_0-m)^i\right)\left(\sum_{i=0}^{k_m}b_{m,i}L_0^i\right)v-\left(\sum_{i=0}^{k_n}b_{n,i}L_0^i\right)\left(\sum_{i=0}^{k_m}b_{m,i}(L_0-n)^i\right)v  \\
= & b_{m,k_m}b_{n,k_n}(nk_m-mk_n)L_0^{k_m+k_n-1}v
+ \text{lower-degree terms w.r.t. } L_0^iv.
\end{split}
\end{equation*}
Hence $k_m=0$ for all $m\in\Z$, i.e., $W_m(v)=a_m(W_0)v$ with
$a_m(W_0) \in \C[W_0]$. In particular, $a_0(W_0)=W_0$.

\begin{clm}\label{claim-L_m(u)}
For any $u \in M$, $i \in \Z_+$ and $m \in \Z$, we have
\begin{equation} \label{relL}
L_m(L_0^iu) = (L_0-m)^iL_mu \text{ and } L_m(W_0^iu) = W_0^iL_mu-imW_0^{i-1}W_mu.
\end{equation}
\end{clm}

We just prove the second relation by induction on $i$.
If $i = 0$, there is nothing to prove.
Now assume the relation holds for $i$,
then we have
\begin{equation*}
\begin{split}
L_m(W_0^{i+1}u) = & [L_m,W_0]W_0^iu+ W_0\big(L_m(W_0^iu)\big) \\
= & -mW_mW_0^iu+W_0\big(W_0^iL_mu-imW_0^{i-1}W_mu\big)  \\
= & W_0^{i+1}L_mu-(i+1)mW_0^iW_mu.
\end{split}
\end{equation*}
Moreover, the second relation is equivalent to
\begin{equation}
L_m(f(W_0)u) = f(W_0)L_mu-mf'(W_0)W_mu\quad \forall\ f(W_0) \in
\C[W_0].
\end{equation}

\begin{clm}\label{deg-a_m}
$\deg(a_m(W_0))=1$ for all $m \in \Z$.
\end{clm}

Now assume $L_mv=g_m(L_0,W_0)v$ for some polynomial $g_m(L_0,W_0) \in \C[L_0,W_0]$.
Then the relation
\begin{equation} \label{relLW}
[L_m,W_n](v) =(n-m)W_{m+n}(v)+\delta_{m+n,0}\frac{m^3-m}{12}c_2v
\end{equation}
implies
\begin{equation*}
\begin{split}
\C[W_0]v \ni & (n-m)W_{m+n}(v)+\delta_{m+n,0}\frac{m^3-m}{12}c_2v  \\
= & L_m(a_n(W_0)v)-W_n(g_m(L_0,W_0)v) \\
= & a_n(W_0)L_mv-ma_n'(W_0)W_mv-g_m(L_0-n,W_0)a_n(W_0)v  \\
= & a_n(W_0)\big(g_m(L_0,W_0)-g_m(L_0-n,W_0)\big)v-ma_n'(W_0)a_m(W_0)v.
\end{split}
\end{equation*}
Now we get $\deg_{L_0}\big(g_m(L_0,W_0)-g_m(L_0-n,W_0)\big)=0$.
It follows that $\deg_{L_0}g_m(L_0,W_0) \leq 1$ for all $m \in \Z$, i.e.,
$g_m(L_0,W_0)=b_m(W_0)L_0+d_m(W_0)$ for some $b_m(W_0), d_m(W_0) \in \C[W_0]$.
Now \eqref{relLW} is equivalent to
\begin{equation} \label{relab}
na_n(W_0)b_m(W_0)-ma_n'(W_0)a_m(W_0) =(n-m)a_{m+n}(W_0) + \delta_{m+n,0}\frac{m^3-m}{12}c_2.
\end{equation}
Taking $m=n$ in \eqref{relab}, we have $b_m(W_0)=a_m'(W_0)$.
Now let $m=-n$ in in \eqref{relab},
then we can obtain $a_n(W_0)a_{-n}(W_0) = W_0^2 - \frac{n^2-1}{12}c_2W_0 +x_n$
for $n \in \Z^*$ and some $x_n \in \C$. This implies that $\deg(a_m(W_0)) \leq 2$ for all $m \in \Z$.

If there are $m \neq n \in \Z$ such that $\deg(a_m(W_0))=\deg(a_n(W_0))=0$ (it is clear that $m+n \neq 0$),
then
\begin{equation}
a_{m+n}(W_0)=(n-m)^{-1}\Big(na_n(W_0)a_m'(W_0)-ma_n'(W_0)a_m(W_0)\Big)=0,
\end{equation}
which is a contradiction.

If $\deg(a_m(W_0))=0$ for some $m \in \Z^*$, then $\deg(a_{2m}(W_0))=1$.
By taking $n=2m$ in \eqref{relab} we get
\begin{equation}
a_{3m}(W_0)=2a_{2m}(W_0)a_m'(W_0)-a_{2m}'(W_0)a_m(W_0)=-a_{2m}'(W_0)a_m(W_0),
\end{equation}
hence $\deg(a_{3m}(W_0))=0$, which is also a contradiction. Now we
have $\deg(a_m(W_0))=1$ for all $m \in \Z$. This claim follows.

\begin{clm}\label{claim-W_m(v)}
$a_m(W_0)=\l^m(W_0-m\a)$ for some $\l\in\C^*$ and $\a\in\C$; and
$c_2=0$.
\end{clm}

From Claim \ref{deg-a_m} we may assume $a_m(W_0)=a_{m,1}W_0+a_{m,0}$
for $a_{m,1},a_{m,0} \in \C$ with $a_{m,1}a_{-m,1} = 1$, $a_{0,1}=1$
and $a_{0,0}=0$. Hence $b_m(W_0)=a_m'(W_0)=a_{m,1}$.

The relation \eqref{relab} implies that
\begin{equation}  \label{am1}
a_{n,1}a_{m,1} =a_{m+n,1},\quad\forall\ m\neq n
\end{equation}
and
\begin{equation}  \label{am0}
na_{n,0}a_{m,1}-ma_{n,1}a_{m,0} =(n-m)a_{m+n,0}+ \delta_{m+n,0}\frac{m^3-m}{12}c_2.
\end{equation}
Then \eqref{am1} gives that $a_{m,1}=a_{1,1}^m$ for all $m \in \Z$.
Now \eqref{am0} gives that $a_{m,0} = m  a_{1,0}a_{1,1}^{m-1}$ and
$c_2=0$.
Set $\l=a_{1,1}$ and $\a=-a_{1,0}/a_{1,1}$, and this claim follows.

\begin{clm}\label{claim-L_m(v)}
$g_n(L_0, W_0)=\l^n(L_0+h_n(W_0)),$ where
$\{h_{n}(t)\}\in\cal{H}_\a$ (see \eqref{calH_a}); and $c_1=0$.
\end{clm}
The equation
\begin{equation}
[L_m,L_n](v) =(n-m)L_{m+n}(v)+\delta_{m+n,0}\frac{m^3-m}{12}c_1v
\end{equation}
is equivalent to
\begin{equation} \label{rela1}
nb_m d_n +nd'_m a_n -md_m b_n -ma_m d'_n =(n-m)d_{m+n}+\delta_{m+n,0}\frac{m^3-m}{12}c_1.
\end{equation}
In particular, by taking $m=-n$ and noticing that $d_0(W_0)=0$, we
have that
\begin{equation}
(d_{-n} a_n + d_na_{-n})' =\frac{1-n^2}{12}c_1.
\end{equation}
For convenience, we denote $F_m=\l^{-m} d_m(W_0)\in\C[W_0]$, then
\eqref{rela1} is equivalent to
\begin{equation}\label{rela2}
n F_n +n F'_m (W_0-n\a) -m F_m -m F'_n (W_0-m\a) =
(n-m)F_{m+n}+\delta_{m+n,0}\frac{m^3-m}{12}c_1.
\end{equation}
It is clear that
\begin{equation*}
\deg_{W_0} F_n \leq \max\{\deg_{W_0} F_{-2}, \deg_{W_0} F_{-1},
\deg_{W_0} F_0, \deg_{W_0} F_1, \deg_{W_0} F_2\}
\end{equation*}
for $n \in \Z$, i.e., the degrees of all $F_n, n\in\Z$ are bounded.
Choose $k\in\Z_+$ such that $\deg F_n\leq k$ for all $n\in\Z$. Let
$f_n$ be the coefficient of $W_0^k$ of $F_n(W_0)$.

%
%
%
%

Checking the coefficients of $W_0^k$ in \eqref{rela2}, we have
\begin{equation}\label{calculator}
(n-km)f_{n}+(kn-m)f_{m}=(n-m)f_{m+n}+\delta_{k,0}\delta_{m+n,0}\frac{m^3-m}{12}c_1.
\end{equation}
\begin{case}\label{k=0} $k=0.$
\end{case}
In this case, we have $$n f_n - m f_m=(n-m)f_{m+n}+
\delta_{m+n,0}\frac{m^3-m}{12}c_1.$$ If $m + n \neq 0$, we have
$(n-m)f_{m+n}=n f_n - m f_m$, hence $f_n=(n-1)f_2-(n-2)f_1$ for $n
\in \Z^*$. For $m =-n\neq 0$, noticing that $f_0=0$, we have
$$\frac{1-n^2}{12}c_1=(n-1)f_2-(n-2)f_1 + (-n-1)f_2-(-n-2)f_1.$$
We get $c_1=0$ and $f_2=2f_1$. Now we have $f_n=nf_1$.

\begin{case}\label{k=1} $k\geq 1.$
\end{case}
Taking $n=2$ and $m=1$ in \eqref{calculator}, we have
\begin{equation*}
f_{3}=(2-k) f_{2} +(2k-1) f_{1}.
\end{equation*}
Taking $n=3$ and $m=1$ in \eqref{calculator}, we have
\begin{equation*}
\begin{split}
2f_{4}= & (3-k) f_{3} +(3k-1) f_{1}  \\
= & (3-k)(2-k) f_{2} +(10k-2k^2-4) f_{1}
\end{split}
\end{equation*}
Taking $n=4$ and $m=1$ in \eqref{calculator}, we have
\begin{equation*}
\begin{split}
6f_{5}= & (4-k) 2f_{4} +2(4k-1) f_{1}  \\
= & (-k^3+9 k^2-26 k+24) f_{2} +2 (k^3-9 k^2+26 k-9) f_{1}
\end{split}
\end{equation*}
Taking $n=3$ and $m=2$ in \eqref{calculator}, we have
\begin{equation*}
\begin{split}
6f_{5}= & 6(3-2k) f_{3} + 6(3k-2) f_{2}  \\
= & 12 (k^2-2 k+2) f_{2} +6(8k-4 k^2-3) f_{1}
\end{split}
\end{equation*}
Calculating the difference of the last two equations we get
$f_{2} = 2 f_{1}$ and by induction $f_{n} = n f_{1}$ for all $n \in
\Z$.

From Proposition \ref{construction}, we see that the family of
polynomials in $W_0$
$$h_{n,k}=nW_0^k-n(n-1)\a\frac{W_0^k-\a^k}{W_0-\a}$$
satisfy the equation \eqref{rela2} with $c_1$ replaced with $0$.
Then there exists $\xi_k\in\C$ such that $\deg(f_n-\xi_k
h_{n,k})\leq k-1$ for all $n\in\Z$. Noticing that the polynomials
$f_n-\xi_k h_{n,k}$ again satisfy the equation \eqref{rela2}.
Replacing $f_n$ with $f_n-\xi_k h_{n,k}$ in the arguments of Case
\ref{k=1} and repeating this process, we may find
$\xi_1,\cdots,\xi_k\in\C$ such that the polynomials
$f_n-\sum_{i=1}^k\xi_ih_{n,i}$ are all constant and satisfy the
equation \eqref{rela2}. Replace $f_n$ with
$f_n-\sum_{i=1}^k\xi_ih_{n,i}$ in the arguments of Case \ref{k=0},
we have $c_1=0$ and
$$f_n-\sum_{i=1}^k\xi_ih_{n,i}=n(f_1-\sum_{i=1}^k\xi_ih_{1,i})=\xi_0h_{n,0}$$
for some $\xi_0\in\C$. As a result, we obtain that
$f_n=\sum_{i=0}^k\xi_ih_{n,i}$, as desired.

Finally, Claim \ref{claim-W_m(u)}, \ref{claim-L_m(u)},
\ref{claim-W_m(v)}, \ref{claim-L_m(v)} together indicate that
$M\cong \Omega(\l,\a,\mathbf{h})$, where
$\mathbf{h}=\{h_n(t)\}\in\mathcal{H}_{\a}$ is as given in Claim
\ref{claim-L_m(v)}. Since simplicity follows from Proposition
\ref{construction} directly, we have completed the proof.
\end{proof}

\begin{rem} We remark that the simple subquotients of the $\W$-module $\Omega_{\W}(\l,\a)$ and
$\Omega(\l,\a,\mathbf{h})$ are all non-weight modules and neither of
the actions of $L_m$ is locally nilpotent in these modules. Hence
they are new classes of modules over $\W$.
\end{rem}

\vskip1cm

%
%
%
\end{document}